\documentclass[11pt]{article}
\usepackage{anysize}\marginsize{3.5cm}{3.5cm}{2cm}{3cm}
\usepackage{amssymb}

\newtheorem{satz}{Satz}[section]

\newtheorem{corollary}[satz]{Corollary}
\newtheorem{definition}[satz]{Definition}
\newtheorem{example}[satz]{Example}

\newtheorem{lemma}[satz]{Lemma}

\newtheorem{remark}[satz]{Remark}

\newtheorem{theorem}[satz]{Theorem}
\newtheorem{conjecture}[satz]{Conjecture}
\newtheorem{thevarthm}[satz]{\varthmname}

\newenvironment{varthm*}[1]{\trivlist\item[]{\bf #1.}\it}{\endtrivlist}

\def\vare{\varepsilon}
\newcommand\eqnref[1]{(\ref{#1})}
\newcommand\lra{\longrightarrow}
\newcommand\bbP{\mathbb P}
\newcommand\bbZ{\mathbb Z}
\newcommand\calo{{\cal O}}
\newcommand\cali{{\cal I}}
\newcommand\newop[2]{\def#1{\mathop{\rm #2}\nolimits}}

\newop\mult{mult}
\newop\upper{upper}
\def\ben{\begin{eqnarray}}
\def\een{\end{eqnarray}}

\begin{document}

\title{Seshadri fibrations of algebraic surfaces}
\author{Wioletta Syzdek and Tomasz Szemberg}
\date{}
\maketitle

\begin{abstract}
   {We refine results of \cite{HwaKeu03} and \cite{SzeTut04} which relate
   local invariants - Seshadri constants - of ample line bundles on surfaces
   to the global geometry - fibration structure. We show that
   the same picture emerges when looking at Seshadri constants measured
   at any finite subset of the given surface.}
\end{abstract}

%*****************************************************************************
\section{Introduction}
\label{intro}

   Seshadri constants were introduced by Demailly \cite{Dem92}
   in an attempt to tackle the Fujita Conjecture \cite{Fuj87}. They quickly
   became an object of independent studies. Nakamaye \cite{Nak03}
   observed that relatively small values of Seshadri constants in
   a general (and hence every) point of an algebraic surface enforce
   a fiber space structure on that surface. The same principle was
   exhibited by Hwang and Keum \cite{HwaKeu03} for varieties of
   arbitrary dimensions. However only in the case of surfaces, there
   are some effective bounds. In \cite{SzeTut04} Tutaj-Gasi\'nska
   and the second author gave a sharp bound for imposing on a surface
   a fibration by Seshadri curves. The cubic
   surface provides an example showing that the obtained bound is in fact optimal.

   In the present paper we study the geometry of surfaces a little bit closer.
   First we show that the cubic surface is the only one for which
   the bound from \cite{SzeTut04} is sharp. Hence we are in a position
   to provide a better bound for all other surfaces.

   Secondly we turn to multiple point Seshadri constants. Somehow
   surprisingly the situation turns out to be similar to that of
   Seshadri constants in a single point. In a sense it is even better,
   as with the number of points increasing our bounds converge
   to the maximal possible value. A similar asymptotic verification
   of the Nagata-Biran Conjecture (see \ref{NBconjecture})
   was obtained before with different methods by Harbourne
   \cite[Theorem I.1]{Har03}.
   We conclude showing that our bounds are optimal
   for arbitrary number of points.

\section{Preliminaries and auxiliary results}
\label{sec:1}

   In this section we recall basic properties of Seshadri constants
   and collect some helpful inequalities.

   First we recall the following definition. Here $r$ is a positive integer.

\begin{definition}
   Let $X$ be a smooth projective variety, let $L$ be an ample line
   bundle on $X$ and let $P_1,\dots, P_r\in X$ be mutually distinct points.
   The $r$-tuple {\rm Seshadri constant} of $L$ at $P_1,\dots, P_r$ is the real number
   $$
      \vare(L;P_1,\dots,P_r)=\inf_{C\cap\{P_1,\dots, P_r\}\neq\emptyset}
      \frac{L.C}{\sum\mult_{P_i}C} \ ,
   $$
   where the infimum is taken over all curves passing
   through at least one of the points $P_1,\dots,P_r$.\\
   We say that a curve $C$ is a {\it Seshadri curve
   for the $r$-tuple $P_1,\dots,P_r$} if
   $$\vare(L;P_1,\dots,P_r)=\frac{L.C}{\sum\mult_{P_i}C}.$$
\end{definition}

   Let $f:Y\lra X$ be the blowing up of $P_1,\dots,P_r\in X$
   with exceptional divisors $E_1,\dots,E_r$. Equivalently
   the Seshadri constant can be computed as
   $$
      \vare(L;P_1,\dots,P_r)=
      \sup\left\{\lambda>0:\, f^*L-\lambda\cdot\sum_{i=1}^rE_i \mbox{ is nef}\right\}.
   $$
   Since the self-intersection of a nef line bundle is non-negative, it follows
   that there is an upper bound:
   $$
      \vare(L;P_1,\dots,P_r)\leq\sqrt[\dim X]{\frac{L^{\dim X}}{r}}=:\vare_{\upper}(L;r).
   $$
   As a function on $X^r$ Seshadri constant
   $\vare(L;\cdot,\dots,\cdot)$ is semi-continuous and has
   the maximal value at a very general point of $X^r$ i.e. on a
   subset of $X^r$ which is the complement of a union of at most
   countably many Zariski closed subsets. We
   abbreviate this maximal value by $\vare(L;r)$.
\begin{remark}\label{scurve}
   We recall that on a surface $X$ a strict inequality
   $$
      \vare(L;P_1,\dots,P_r)<\vare_{\upper}(L;r),
   $$
   implies via the real valued Nakai-Moishezon criterion \cite{CamPet90}
   that there is a Seshadri curve
   for the $r$-tuple $P_1,\dots,P_r$. Such a curve can be assumed to be reduced
   and irreducible.\\
   In particular if
   $$
      \vare(L;r)<\vare_{\upper}(L;r),
   $$
   then there is a Seshadri curve through every $r$-tuple of points on $X$.
\end{remark}
   The following lemma which is due to Xu \cite[Lemma 1]{Xu95} will
   be used to estimate the self-intersection of Seshadri curves.
\begin{lemma}\label{xu}
  Let $X$ be a smooth projective surface, let $(C_t,(P_1)_t,\dots,(P_r)_t)_{t\in T}$ be
  a non-trivial
  one parameter family of pointed reduced and irreducible curves on $X$ and let $m_i$
  be positive integers such that $\mult_{(P_i)_t}C_t\geq m_i$ for all $i=1,\dots, r$. Then
  $$
     \begin{array}{lcl}
     \mbox{for } r=1 \mbox{ and } m_1\geq 2 && C_t^2\geq m_1(m_1-1)+1 \mbox{ and }\\
     &&\\
     \mbox{for } r\geq 2 && C_t^2\geq \sum_{i=1}^r m_i^2 - \min\{m_1,\dots,m_r\}.
     \end{array}
  $$
\end{lemma}
   The second lemma was obtained by K\"uchle in \cite{Kue96} and is purely numerical.
\begin{lemma}\label{kuechle}
   Let $r\geq 2$ and $m_1,\dots ,m_r\in {\bbZ}$ be
   integers with $m_1\geq\dots\geq m_r\ge 1$ and $m_1\ge 2$. Then we
   have
   $$(r+1)\sum_{i=1}^r m_i^2 > \left(\sum_{i=1}^r m_i\right)^2+m_r(r+1).$$
\end{lemma}

\section{Single point Seshadri constants and fibrations}
\label{sec:2}
   Recall that in the case of algebraic surfaces Hwang and Keum proved
   the following result \cite[Theorem 2]{HwaKeu03}.

\begin{theorem}[Hwang-Keum]\label{hk}
   Let $X$ be a projective surface and $L$ an ample line bundle on
   $X$ with
   $$
     \vare(L;1)<\sqrt{\frac34}\cdot \vare_{\upper}(L;1).
   $$
   Then there is a fibration of $X$ whose fibers are Seshadri curves of
   $L$.
\end{theorem}
   It was shown in \cite{SzeTut04} that if $X\subset\bbP^3$ is a smooth
   cubic surface and $L=\calo_X(1)$, then
   $$
      \vare(L;1)=\sqrt{\frac34}\cdot\vare_{\upper}(L;1)=\frac32
   $$
   and that $X$ is not fibered by Seshadri curves. This means that the factor
   of $\sqrt{\frac34}$ in the above cannot be improved in general.
   Here we show however that the cubic is the only example of this kind.
\begin{theorem}\label{itiscubic}
   Let $X$ be a smooth projective surface and $L$ an ample line bundle
   on $X$ such that
   $$
      \vare(L;1)=\sqrt{\frac34}\cdot\vare_{\upper}(L;1).
   $$
   If $X$ is not fibered by Seshadri curves, then $X$ is the cubic surface in $\bbP^3$
   and $L$ is the hyperplane bundle.
\end{theorem}
\begin{proof}
   For any $x\in X$ let $D_x$ be an irreducible and reduced Seshadri curve for $x$
   (see Remark \ref{scurve}) with multiplicity $m_x=\mult_x(D_x)$ in that point.
   Further, let $X_0\subset X$ be an open and dense subset of $X$ on which the
   multiplicity of Seshadri curves is constant equal $m$.

   The proof goes in several steps. Here is the outline. First with Hodge Index
   we show that $m=2$ (note that for the cubic surface Seshadri constants are
   computed by tangent sections). In the second step we use Kodaira-Spencer map
   to prove that the curves $D_x$ are rational. Finally with some ad hoc arguments
   we conclude that $D_x$ are members of a linear system embedding $X$ as
   a cubic in $\bbP^3$.

   Let $x\in X_0$ be given. If $m=1$, then by Hodge Index Theorem we have
   $$
      \frac34L^2=(L.D_x)^2\geq L^2D_x^2,
   $$
   which implies $D_x^2=0$ and, as in the proof of \cite[Theorem]{SzeTut04},
   that there is a Seshadri fibration on $X$.

   Hence $m\geq 2$ and by Lemma \ref{xu} we have
   $$
      D_x^2\geq m(m-1)+1.
   $$
   Combining this inequality with our numerical assumptions on $L$
   and applying Hodge Index we obtain
   $$
      \frac34\geq\frac{m(m-1)+1}{m^2},
   $$
   which is equivalent to $(\frac12m-1)^2\leq 0$. This implies $m=2$ and
   since then there is an equality in the Hodge Index, we conclude that
   $D_x^2=3$ and $\calo_X(D_x)$ is ample, as it is numerically equivalent to
   some rational multiple of $L$.

   Now we turn to the rationality of $D_x$. Since the curves $D_x$ are reduced
   and $m=2$ it can happen only for finitely many points $y$ that
   $D_x=D_y$. This means that we have a two-parameter family of Seshadri curves.
   We fix $x_0$ in the interior of $X_0$ and a sufficiently small disk $\Delta$
   so that $\Delta\times\Delta\subset X_0$ is a neighborhood of $x_0$. Then
   the deformation $\left(D_{x_{(t,s)}}, x_{(t,s)}\right)_{\Delta\times\Delta}$
   of the pointed Seshadri curve $(D_0,x_o)$ determines
   a non-degenerate Kodaira-Spencer map
   $$
      \rho:T_0\Delta\times T_0\Delta\lra H^0(D_{x_0},\calo_{D_{x_0}}(D_{x_0})).
   $$
   We abbreviate $D=D_{x_0}$.
   As in \cite[Corollary 1.2]{EinLaz92} we conclude that
   there are two independent sections $\rho(\frac{d}{dt})$ and
   $\rho(\frac{d}{ds})$ in $H^0(D,\calo_{D}(D)\otimes\cali_{x_0})$,
   where $\cali_{x_0}$ is the maximal ideal.

   Let $f:Y\lra X$ be the blowing up of $X$ at $x_0$ with the
   exceptional divisor $E$ and let
   $D'=f^*D-2E$ be the proper transform of $D$. By the projection formula
   we have that $M=f^*\calo_D(D)\otimes\calo_Y(-E)$ has at least two global
   sections. On the other hand $\deg M=1$, which implies that $D'$ is rational
   and hence so is $D$.

   Thus $X$ is rationally connected (any two curves $D_x$ and $D_y$ intersect),
   so it is a rational surface. In particular all curves $D_x$ are linearly equivalent
   and equivalent to $L$. We denote the linear system generated by
   the curves $D_x$ simply by $|D|$. Note that we have thus
   obtained a complete family of Seshadri curves. Since for a very general point
   $x\in X$ there exists in $|D|$ a  curve singular at $x$, we
   find such a curve (possibly reducible) for every point on the surface.

   Now we show that the linear system $|D|$ is base point free. To this end let
   $y\in X$ be fixed. There is a curve $D_y\in|D|$ with $\mult_yD_y\geq 2$.
   Let $y_1$ be a general smooth point of $D_y$. Again, there exists an irreducible curve
   $D_{y_1}\in|D|$ with $\mult_{y_1}D_{y_1}=2$.
   We claim that this curve does not go through y.
   Indeed, since $D_y$ and $D_{y_1}$ have no common components and their
   intersection product is $3$ they cannot meet each other in singular points.

   Taking into account that $|D|$ is ample and base point free, the image of
   the induced mapping of $X$ has dimension $2$. Bertini's theorem tells us
   that a general member of $|D|$ is smooth and irreducible. Since we have already
   identified a two-parameter family of singular divisors in $|D|$, it implies
   that $\dim |D|\geq 3$.

   Finally we prove that the linear system $D$ defines an embedding of $X$.
   First we show that $|D|$ separates points. This goes similarly as
   the global generation. Let $y_1$ and $y_2$ be two different points on $X$
   and let $D_{y_1}$ and $D_{y_2}$ be the corresponding Singular curves in $|D|$.
   If $y_1\notin D_{y_2}$, then we are done. If $y_1\in D_{y_2}$
   it cannot be $y_2\in D_{y_1}$
   at the same time, otherwise it contradicts $D_{y_1}.D_{y_2}=3$, so we are done again.

   Now, let $x\in X$ and $\overrightarrow{v}\in T_xX$ be fixed and let $D_x$ be the singular curve
   at $x$ in $|D|$. If $\overrightarrow{v}$ is not tangent to $D_x$, then we are done.
   Suppose that all curves in the system $|D\otimes\cali_x|$ have $\overrightarrow{v}$ as
   a tangent vector, then all these curves intersect $D_x$ only in $x$, as the
   intersection multiplicity already at that point is equal $3$ or they have
   a common component with $D_y$. On the other hand let $x_1\in D_x$ and $x_2\notin D_x$ be general.
   Dimension count shows that there is an irreducible curve
   $C\in|D\otimes\cali_x\otimes\cali_{x_1}\otimes\cali_{x_2}|$.
   Such a curve has no common components with $D_x$ and would have intersection
   multiplicity $\geq 4$, a contradiction.

   Summing up, we have shown that the linear system $|D|$ embeds $X$ as a surface
   of degree $3$ in a projective space. Since a complete embedding is non-degenerate,
   the image of $X$ must be a smooth cubic in $\bbP^3$.
\end{proof}
   Now we are in a position to improve the bound in Theorem \ref{hk}.
\begin{corollary}
   Let $X$ be a smooth projective surface and $L$ an ample line bundle on $X$.
   If
   $$
      \vare(L;1)<\sqrt{\frac79}\cdot\vare_{\upper}(L;1),
   $$
   then
   \begin{itemize}
      \item[a)] either $X$ is a cubic in $\bbP^3$ and $L=\calo_X(1)$,
      \item[b)] or $X$ is fibered by Seshadri curves.
   \end{itemize}
\end{corollary}
\begin{proof}
   We assume to the contrary that $X$ is neither a cubic nor is it fibered by Seshadri
   curves. From the proof of Theorem \ref{itiscubic} it follows immediately that
   the multiplicity of Seshadri curve at a general point of $X$ is at least $3$.
   On the other hand combining Lemma \ref{xu} with the Hodge Index and our
   numerical assumption, it is elementary to see that this is impossible.
\end{proof}

\section{Multiple point Seshadri constants and fibrations}
   We now pass to the second part of our paper and investigate the relationship
   between multiple point Seshadri constants and fibrations by Seshadri curves.
\begin{theorem}\label{multipoint}
   Let $X$ be a smooth projective surface, $L$ an ample line
   bundle on $X$ and $r\geq 2$ a fixed integer. If
   \ben\label{star}
      \vare(L;r)<\sqrt{\frac{r-1}{r}}\cdot\vare_{\upper}(L;r)
   \een
   then there exists a fibration $f:X\lra D$ over a curve $D$ such
   that given $P_1, \dots, P_r\in X$ very general, the fiber
   $f^{-1}(f(P_i))$ computes $\vare(L;P_1,\dots,P_r)$ for arbitrary
   $i=1,\dots, r$.\\
   Moreover the factor $\sqrt{\frac{r-1}r}$ is optimal for every $r$.
\end{theorem}
\begin{proof}
   Let $P_1,\dots,P_r\in X$ be very general. Since
   $\vare(L;P_1,\dots,P_r)$ is not maximal, there exists a Seshadri curve
   $C_{P_1,\dots,P_r}$.  Moving the
   points around we obtain a non-trivial family $C_t=C_{(P_1)_t,\dots,(P_r)_t}$ of such curves.
   Let $m_1\geq \dots \geq m_r$ be non-negative integers such that
   $\mult_{(P_i)_t}C_t= m_i$ for the general member
   $C_t$ of the family.

   We proceed by induction on $r$ and begin with $r=2$ (note that
   our assertion is empty for $r=1$ and we cannot use the Hwang-Keum
   result as the first step of the induction).

   First we assume that $m_1\geq m_2\geq 1$. From Lemma
   \ref{xu} we obtain
   $$
      (m_1^2+m_2^2-m_2)\cdot L^2\leq (C_t)^2\cdot L^2.
   $$
   On the other hand, by our numerical assumptions we have
   $$
      (L.C_t)^2<(m_1+m_2)^2\cdot \frac14\cdot L^2.
   $$
   Both inequalities can be written in one line thanks to the
   Hodge Index Theorem:
   $$
      m_1^2+m_2^2-m_2<\frac14(m_1+m_2)^2.
   $$
   This is equivalent to
   $$
      2(m_1^2+m_2^2)+(m_1-m_2)^2 < 4m_2,
   $$
   which is false as $m_1\geq m_2\geq 1$.

   If $m_2=0$, then by the assumption of the Theorem
   $$
      \frac{L.C_t}{m_1}<\sqrt{\frac14 L^2}<\sqrt{\frac34 L^2}=\sqrt{\frac34}\vare_{\upper}(L;1)
   $$
   and in this case our assertion follows from Theorem \ref{hk}.

   For the induction step assume that the number of points $r$ is
   at least $3$. There are the following possibilities:\\
   (a) $m_1\geq \dots \geq m_r\geq 1$ and $m_1\geq 2$ or\\
   (b) $m_1=\dots =m_r=1$ or\\
   (c) $m_r$=0.

   In case (a) the Hodge Index
   Theorem together with Lemma \ref{xu} give:
   $$
      \frac{\sum m_i^2-m_r}{(\sum m_i)^2}L^2\leq
      \frac{L^2\cdot (C_t)^2}{(\sum m_i)^2}\leq
      \frac{(L.C_t)^2}{(\sum m_i)^2}<
      \frac{r-1}{r^2}L^2.
   $$
   Hence by Lemma \ref{kuechle} we obtain
   $$
     \sum m_i^2-m_r<
     \frac{r-1}{r^2}\left(\sum m_i\right)^2<
     \frac{(r-1)(r+1)}{r^2}\left(\sum m_i^2-m_r\right),
   $$
   a contradiction.

   Case (b) is also immediately excluded as $(C_t)^2\geq r-1$ by Lemma
   \ref{xu} and thus
   $$
      \frac{L.C_t}{\sum m_i}=
      \frac{L.C_t}{r}\geq
      \frac{\sqrt{r-1}}{r}\sqrt{L^2}
   $$
   by Hodge Index Theorem contradicting our assumption
   \eqnref{star}.

   In the last case (c) we have
   $$
     \frac{L.C_t}{\sum_{i=1}^{r-1} m_i}=
     \frac{L.C_t}{\sum_{i=1}^r m_i}<
     \sqrt{\frac{r-1}{r^2} L^2}<
     \sqrt{\frac{r-2}{(r-1)^2}L^2},
   $$
   where the first inequality is just our assumption \eqnref{star}
   and the second holds as $r\geq 3$.
   This shows that the assumptions of our Theorem are satisfied
   for $r-1$ and we conclude by induction.
\end{proof}
The following example shows that our bound is optimal.
\begin{example}
   Let $X=\bbP^2$, let $L=\calo_{\bbP^2}(1)$ and let $r=2$. Then
   the line through two given points $P_1, P_2$ computes
   $\vare(L;P_1,P_2)=\frac12=\sqrt{\frac{r-1}{r}}\cdot\vare_{\upper}(L;2)$ and
   there is no fibration on $\bbP^2$.

   More generally, let $r$ be given and
   let $X$ be a rational normal scroll in $\bbP^r$ and
   $L=\calo_X(1)$. The scroll is of course fibered but the curves in the ruling
   are not the Seshadri curves. To see this let $P_1,\dots,P_r\in X$ be points in
   general position. Then obviously for a fiber $F$ of the ruling
   passing through one or more points of the set $P_1,\dots,P_r$ we have
   $$
      \frac{L.F}{\sum\mult_{P_i} F}=1.
   $$
   On the other hand $r$ points span a hyperplane in $P^r$ i.e.
   there is a curve $C\in|L|$ passing through all of them with
   Seshadri quotient
   $$
      \frac{L.C}{\sum\mult_{P_i} C}=\frac{r-1}{r}=\sqrt{\frac{r-1}{r}}\cdot\vare_{\upper}(L;r)<1.
   $$
   Now, Bertini Theorem implies that $C$ is irreducible for general
   $P_1,\dots,P_r$. Let $D\neq C$ be an irreducible curve passing
   through at least one of the points $P_1,\dots,P_r$. Then
   $$L.D=L.C\geq \sum_{i=1}^r \mult_{P_i}D\cdot\mult_{P_i}C\geq
     \sum_{i=1}^r \mult_{P_i}D.$$
   This shows that the hyperplane section is the only Seshadri curve on $X$
   (In fact this is always the case if there is a Seshadri curve in $|L|$ itself).
   So $X$ is not fibered by the Seshadri curves in this case.
\end{example}
   Our results deal with situations when Seshadri constants are relatively small
   related to the upper bound. In fact it is conjectured that Seshadri
   constants at sufficiently many points are always maximal.
   Below we formulate this conjecture more exactly, it
   interpolates on the well known Nagata conjecture,
   see \cite{Sze01} for an effective statement, background and equivalent
   formulations.
\begin{conjecture}[Nagata-Biran]\label{NBconjecture}
   Let $X$ be a smooth projective variety and $L$ an ample line
   bundle on $X$. Then there exists a number $r_0$ (depending on
   $X$ and $L$) such that for all $r\geq r_0$
   $$
      \vare(L;r)=\vare_{\upper}(L;r).
   $$
\end{conjecture}
   Theorem \ref{multipoint} can be viewed as an asymptotic
   confirmation of the above conjecture. A similar result was
   obtained with different methods by Harbourne \cite{Har03}.
\begin{corollary}
   If a surface $X$ admits no fibration over a curve (e.g. a
   general surface of general type), then
   $$
      \vare(L;r)\geq\sqrt{\frac{r-1}r}\cdot\vare_{\upper}(L;r).
   $$
   In particular the Nagata-Biran conjecture holds on $X$ asymptotically.
\end{corollary}

\paragraph*{Acknowledgement.}
   We would like to thank H. Esnault and E. Viehweg for inviting us to
   the Mathematics Department in Essen where
   most of this work was prepared.
   Financial support was made possible by the Leibnitz Preis
   of Esnault and Viehweg. Further we would
   like to thank Lawrence Ein who brought us to idea of looking at
   Theorem \ref{itiscubic}. Finally we would like to thank both
   referees for helpful remarks and comments.

   The second author was partially supported by
   KBN grant 1 P03 A 008 28.

%*****************************************************************************

%*****************************************************************************

\bigskip
\small
   Wioletta Syzdek, Tomasz Szemberg

   Instytut Matematyki AP,
   PL-30-084 Krak\'ow, Poland

\nopagebreak
   E-mail: \texttt{syzdek@ap.krakow.pl},

\nopagebreak
   E-mail: \texttt{szemberg@ap.krakow.pl}
\end{document}